%% file: main.tex
\documentclass[12pt]{article}
\usepackage{amsmath,amsthm,amssymb,amsfonts}
\usepackage[numbers,sort&compress]{natbib}
\usepackage{color,colordvi}
\usepackage{soul}
\usepackage{fullpage}
\usepackage[letterpaper,top=2cm,bottom=2cm,left=3cm,right=3cm,marginparwidth=2cm]{geometry}

\usepackage{amsmath,amsfonts,amsthm,mathtools,bm}
\usepackage{graphicx}
\usepackage[colorlinks=true, allcolors=blue]{hyperref}
\usepackage[capitalise]{cleveref}
\usepackage{standalone}
\usepackage{indentfirst}
\usepackage[affil-it]{authblk}
\usepackage{diagbox}
\usepackage{changes}
\usepackage{enumitem}
\usepackage{bbm}
\usepackage{orcidlink}

\newtheorem{theorem}{Theorem}[section]
\newtheorem{lemma}[theorem]{Lemma}
\newtheorem{corollary}[theorem]{Corollary}

\newtheorem{example}[theorem]{Example}

\newtheorem{question}[theorem]{Question}

\newtheorem{definition}[theorem]{Definition}

\newtheorem*{claim*}{Claim}

\newcommand{\textotherwise}{\text{otherwise}}

\newcommand{\dd}{\mathrm{d}}

\newcommand{\RR}{\mathbb{R}}

\newcommand{\cD}{\mathcal{D}}

\newcommand{\cP}{\mathcal{P}}

\newcommand{\bigO}{\mathcal{O}}

\DeclareMathOperator{\dist}{\mathrm{dist}}
\DeclareMathOperator{\Tr}{\mathrm{Tr}}
\DeclareMathOperator{\con}{\mathrm{con}}
\DeclareMathOperator{\Expect}{\mathbb{E}}
\DeclareMathOperator{\crossing}{\mathrm{cr}}

\DeclarePairedDelimiter{\card}{\lvert}{\rvert}
\DeclarePairedDelimiter{\set}{\lbrace}{\rbrace}
\DeclarePairedDelimiter{\norm}{\lVert}{\rVert}
\DeclarePairedDelimiter{\paren}{\lparen}{\rparen}
\DeclarePairedDelimiter{\lrbracket}{\lbrack}{\rbrack}
\DeclarePairedDelimiter{\lrbrace}{\lbrace}{\rbrace}
\DeclarePairedDelimiter{\floor}{\lfloor}{\rfloor}

\title{Upper bounds of Steklov eigenvalues on graphs}

\author{Huiqiu Lin\footnote{email: huiqiulin@126.com}}
\author{Lianping Liu\footnote{email: y30231283@mail.ecust.edu.cn}}
\author{Zhe You\footnote{email: y30231280@mail.ecust.edu.cn}}
\author{Da Zhao\footnote{email: zhaoda@ecust.edu.cn}~\orcidlink{0000-0002-9582-0778}}
\affil{School of Mathematics, East China University of Science and Technology, Shanghai 200237, China.}

\date{}

\begin{document}
\maketitle

\begin{abstract} 
    %The Steklov eigenvalue problem was introduced over a century ago, and its discrete form recently attracted interest.
    Let $\Delta$ and $B$ be the maximum vertex degree and a subset of vertices in a graph $G$ respectively. 
    In this paper, we study  the first (non-trivial) Steklov eigenvalue $\sigma_2$ of $G$ with boundary $B$. 
    Using metrical deformation via flows, we first show that $\sigma_2 = \mathcal{O}\left(\frac{\Delta(g+1)^3}{|B|}\right)$ for graphs of orientable genus $g$ if $|B| \geq \max\{3 \sqrt{g},|V|^{\frac{1}{4} + \epsilon}, 9\}$ for some $\epsilon > 0$. 
    This can be seen as a discrete analogue of Karpukhin's bound.
    %, that is, $\sigma_2 \leq \frac{2\pi (g+1)}{|\partial \Omega|}$ for compact surfaces $\Omega$ of genus $g$ with boundary $\partial \Omega$.
    Secondly, we prove that $\sigma_2 \leq \frac{8\Delta+4X}{|B|}$ based on planar crossing number $X$.
    Thirdly, we show that $\sigma_2 \leq \frac{|B|}{|B|-1} \cdot \delta_B$, where $\delta_B$ denotes the minimum degree for boundary vertices in $B$.
    At last, we compare several upper bounds on Laplacian eigenvalues and Steklov eigenvalues.
\end{abstract}


% \tableofcontents

\section{Introduction}
The spectra of linear operators are important objects in Riemannian geometry, partial differential equations, graph theory, mathematical physics and so on. The Steklov problem is a classical eigenvalue problem in  spectral geometry exhibiting interesting interactions between analysis and geometry.
Let $\Omega$ be a compact smooth orientable Riemannian manifold with boundary $\Sigma=\partial\Omega$. Consider the Dirichlet-to-Neumann operator $\cD : C^\infty(\Sigma) \to C^\infty(\Sigma)$ defined by $\cD f := \dfrac{\partial \hat{f}}{\partial n}$, where $n$ is the outward normal along the boundary and $\hat{f} \in C^\infty(\Omega)$ is the unique harmonic extension of $f$ into $\Omega$. 
Therefore, the Steklov problem on a compact Riemannian manifold $\Omega$ with boundary $\Sigma$ is known to be
$$
\begin{cases}\Delta \hat{f}=0 & \text { in } \Omega \\ 
\cD f=\sigma f & \text { on } \Sigma.\end{cases}
$$
Here, $\Delta$ denotes the Laplace–Beltrami operator. The spectrum of $\cD$ is discrete and can be ordered as
\begin{align*}
    0 = \sigma_1 \leq \sigma_2 \leq \cdots.
\end{align*}
These eigenvalues are called Steklov eigenvalues. 
We remark that some literature counts $\sigma_k$ starting from $0$. Some results in continuous situations were given in \cite{weinstock_inequalities_1954,escobar_geometry_1997,escobar_isoperimetric_1999,brock_isoperimetric_2001,wang2009sharp,karpukhin2014multiplicity,girouard2016steklov,yang2017estimates,fall2017profile,fraser2019shape,chen2024upper,girouard2024large}.

A graph $G = (V,E)$ is a tuple of the vertex set $V$ and the edge set $E$. 
The boundary of the graph, denoted by $B$, is chosen as a non-empty subset of $V$, and we use $\Omega = V \backslash B$ for the set of rest vertices. 
We assume that $|B| \geq 2$ to exclude the trivial cases.

We consider the Steklov problem on the pair $(G, B)$. 
For a real function $f \in \RR^V$ on $V$, the discrete Laplacian operator $\Delta$ can be defined by
\begin{align*}
    (\Delta f)(x) = \sum_{(x,y) \in E} (f(x) - f(y)).
\end{align*}
The Steklov eigenvalue problem on graphs can be viewed as the following equations.
\begin{equation*}
    \begin{cases}
        \Delta f(x)=0, & x\in \Omega, \\ 
        \mathcal{D} f(x)=\sigma f(x), & x\in B .
    \end{cases}
\end{equation*}
$\sigma \in \mathbb{R}$ is called the Steklov eigenvalue of the graph with boundary $B$.
For any function $0 \neq f\in \RR^{V}$, the Rayleigh quotient of $f$ is defined as
\begin{align*}
    R(f) = \frac{\sum_{(x,y) \in E} (f(x)-f(y))^2 }{\sum_{x \in B} f^2(x)},
\end{align*}
where the right hand side is understood as $+\infty$ if $f|_{B}=0$.
The variational characterizations of the Steklov eigenvalues are given as
\begin{align}
    \sigma_k &= \min_{W \subset \RR^V, \dim W = k} \max_{f \in W} R(f), \nonumber \\
    \sigma_k &= \min_{W \subset \RR^V, \dim W = k-1, W \perp 1_{B}} \max_{f \in W} R(f), \label{eqn:Steklov_minmax2}
\end{align}
where $1_{B}$ is the characteristic function on $B$.

The idea of getting the upper bound for Escobar Cheeger-type constant in~\cite{hua_first_2017} motivates the authors to obtain the upper bounds for the Steklov eigenvalues for graphs. 
He and Hua~\cite{he_upper_2022} showed that $\sigma_2\leq \frac{8(\Delta-1)}{|V|+2}$ and $\sigma_2\leq \frac{2}{L}$ where $\Delta$ and $L$ respectively denote the maximum degree and diameter of $G$. They further got that $\sigma_k\leq \frac{8(\Delta-1)^2(k-1)}{|B|}.$
For finite subgraphs in integer lattices, an upper bound of $\sigma_2$ was given by Han and Hua~\cite{han_steklov_2023}. Some other results can be found in \cite{perrin_isoperimetric_2021,perrin_lower_2019,tschanz2022upper}.
Very recently, the upper bound of the first Steklov eigenvalue in planar graphs is obtained by using circle packing and conformal mapping.

\begin{theorem}[{\cite[Theorem 1.1]{lin_first_2024}}]
    Let $G$ be a planar graph with boundary $B$ such that the vertex degree is bounded by $\Delta$. 
    Then
    \begin{align*}
        \sigma_2 \leq \frac{8\Delta}{\card{B}}. 
    \end{align*}
\end{theorem}

A natural question is to further consider Steklov eigenvalues of the graphs which can be embedded into a surface. 
In this paper, we generalize the result of planar graphs in two ways.

\begin{theorem}\label{thm:lambda2_genus}
    Let $G$ be a graph with (orientable) genus $g$  and boundary $B$ such that the vertex degree is bounded by $\Delta$.  
    If $\card{B} \geq \max\set*{3 \sqrt{g},|V|^{\frac{1}{4} + \epsilon},9}$ for some $\epsilon > 0$, then
    \begin{align}\label{eq:genus_lambda2}
        \sigma_2 \leq \bigO\paren*{\frac{\Delta (g+1)^3}{|B|}}.
    \end{align}
\end{theorem}

%\begin{theorem}\label{thm:lambda2_genus}
    %Let $G$ be a graph with (orientable) genus $g$  and boundary $B$ such that the vertex degree is bounded above by $\Delta$.  
    %If $\card{B} \geq \max\set*{3 \sqrt{g},|V|^{\frac{1}{4} + \epsilon}}$ for some $\epsilon > 0$, then
    %\begin{align}\label{eq:genus_lambda2}
        %\sigma_2 \leq \bigO\paren*{\frac{\Delta g(\log g)^2}{|B|}}.
    %\end{align}
%\end{theorem}

The (orientable) genus of a graph $G$ is the minimum genus of an (orientable) surface in which the graph can be embedded. Compare with Karpukhin's bound~\cite[Theorem 1.4]{karpukhin2017BoundsLaplaceSteklov} for compact orientable Riemann surfaces $\Omega$ of genus $g$ with boundary $\partial \Omega$, namely
\begin{align}
    \sigma_k \leq \frac{2\pi(g + b + k - 2)}{\card{\partial \Omega}},
\end{align}
where $b$ is the number of boundary components.
In particular
\begin{align}\label{eq:lambda2_genus_continuous}
    \sigma_2 \leq \frac{2\pi(g + 1)}{\card{\partial \Omega}}.
\end{align}
for compact orientable Riemann surface of genus $g$ with connected boundary. 
Note that the growth of $g$ in~\cref{eq:genus_lambda2} is cubic while it is linear in~\cref{eq:lambda2_genus_continuous}.

For graphs with crossing number $X$, which is defined as the smallest number of pairwise crossings of edges among all drawings of $G$ in the plane, we can get a more precise upper bound on $\sigma_2$.
\begin{theorem}\label{thm:lambda2_planar-crossing}
    Let $G$ be a graph with boundary $B$ such that the degree is bounded by $\Delta$ and that the crossing number is $X$. 
    Then
    \begin{align}\label{eq:planar-crossing_lambda2}
        \sigma_2 \leq \frac{8\Delta + 4X}{\card{B}}. 
    \end{align}
\end{theorem}

We also obtain an upper bound of first Steklov eigenvalue by the minimum degree of boundary vertices. 

\begin{theorem}\label{thm:bound_min_deg}
    Let $G$ be a graph with boundary $B$ such that the minimum degree for boundary vertices is $\delta_{B} = \min_{x \in B} \deg(x)$. 
    Then
    \begin{align*}
        \sigma_2 \leq \frac{\card{B}}{\card{B} - 1} \cdot \delta_{B}. 
    \end{align*}
\end{theorem}

For higher Steklov eigenvalues, we have the following bound. 

\begin{theorem}\label{thm:interlace_bound}
    Let $G = (V,E)$ be a graph with boundary $B$. 
    Let $L$ be the Laplacian matrix of $G$, and let $N$ be the principal submatrix of $L$ indexed by $B$. 
    Suppose that the eigenvalues of $N$ are
    \begin{align*}
        \mu_1(N) \leq \mu_2(N) \leq \cdots \leq\mu_{\card{B}}(N).
    \end{align*}
    Then
    \begin{align*}
        \sigma_k &\leq \mu_k(N), \quad k = 1,2, \ldots, \card{B}.
    \end{align*}
\end{theorem}

\begin{corollary}\label{thm:degree_bound}
    Let $G = (V,E)$ be a graph with boundary $B$ such that the boundary vertices form an independent set of $G$. 
    Suppose that the degree sequence of the boundary vertices is $d_1 \leq d_2 \leq \cdots \leq d_{\card{B}}$. 
    Then
    \begin{align*}
        \sigma_k \leq d_k,\quad k = 1,2, \ldots, \card{B}. 
    \end{align*}
\end{corollary}

\section{Proofs}

\subsection{Proof of~\texorpdfstring{\cref{thm:lambda2_genus}}{Theorem }}

Before proving~\cref{thm:lambda2_genus}, we first need some notation and definitions.

Given two expressions $A$ and $B$ (possibly depending on a number of parameters), we write $A = \bigO_{p}(B)$ to express that $A \leq C_p B$ for some constant $C_p > 0$ which is independent of the variables in $B$ but may depend on $p$. 
If the constant $C_p$ is universal, then we omit $p$ in the notation, namely we write $A = \bigO(B)$. 
Similarly, $A = \Omega(B)$ implies that $A \geq CB$ for some $C > 0$. 
We also write $A \lesssim B$ as a synonym for $A=\bigO(B)$.

Let $G = (V,E)$ be an undirected graph with boundary $B$. 
For every pair $u,v \in B$, let $\cP_{uv}$ be the set of paths between $u$ and $v$ in $G$. 
Denote $\cP = \bigcup_{u,v\in B} \cP_{uv}$. 
A boundary flow in $G$ is a map $F : \mathcal{P} \longrightarrow \mathbb{R}_+$. 

\begin{definition}[Vertex congestion]
    For every vertex $v \in V$, the value
    \begin{align*}
        C_F(v) = \sum_{p \in \cP : v \in p} F(p)
    \end{align*}
    is the vertex congestion of $F$ at $v$. 
    For $p \geq1$, we define the vertex $p$-congestion of $F$ by
    \begin{align*}
        \con_p(F) = \paren*{\sum_{v \in V} C_F(v)^p}^{1/p}.
    \end{align*}
\end{definition}

\begin{definition}[Integral boundary flow and unit $H$-boundary flow]
    We say that $F$ is an integral boundary flow if $\card{\set{p \in \cP_{uv}:F(p)>0}}\leq 1$ for every $u,v \in B $. 
    Given a demand graph $H =(U,D)$, we say that $F$ is a unit $H$-boundary flow if there exists an injective map $g : U  \longrightarrow B$ such that for every $(i, j) \in D$, we have $\sum_{p \in \cP_{g(i)g(j)}} F(p) = 1$, and $F(p) = 0$ if $p \notin \cup_{(i,j)\in D} \cP_{g(i)g(j)}$. 
    An integral $H$-boundary flow is a unit $H$-boundary flow that is also integral.
\end{definition}

We give a lemma which is similar to the lemma in~\cite{biswal2010eigenvalue}.

\begin{lemma}\label{lem:integral-flow}
    Let $G=(V,E)$ be a graph with boundary $B$. 
    Suppose $H =(U,D)$ is a demand graph and $F$ is a unit $H$-boundary flow in $G$. 
    Then there exists an integral $H$-boundary flow $F^*$ such that
    \begin{align*}
        \con_2(F^*) \leq \con_2(F) + \sqrt{\con_1(F)}.
    \end{align*}
    Furthermore, if $\card{B} \geq \card{V}^{\frac{1}{4} + \epsilon}$ for some $\epsilon > 0$, then
    \begin{align*}
        \con_2(F^*) \leq 2\con_2(F) + \card{B}^{\frac{2}{1 + 4 \epsilon}}.
    \end{align*}
\end{lemma}

\begin{proof}
    For a map $F:\mathcal{P} \longrightarrow \mathbb{R}_+$ and vertices $x, u, v \in V$ such that $u, v\in B$, let $F_{uv}(x) =\sum_{x \in p \in \mathcal{P}_{uv}} F(p)$. 
    Define a random flow $F^*$ as follows: 
    for each demand pair $u,v \in B$, pick one path $p$ among $\cP_{uv}$ independently with probability $F(p)$ (Here $F(p)$ with $p \in \cP_{uv}$ is indeed a probability distribution since for each demand pair $(i,j) \in D$, we can find $u,v \in B$ such that $u = g(i),  v = g(j)$ and $\sum_{p \in \cP_{uv}} F(p) = 1$). 
    Set $F^*(p) = 1$ for each of the selected paths, and $F^*(p) = 0$ for all other paths. 
    It is clear that $F^*$ is an integral $H$-boundary flow. 
    Then
    \begin{align*}
        \Expect\lrbracket*{\con_2(F^*)^2} &= \Expect\lrbracket*{ \sum_{x \in V} \paren*{\sum_{u,v \in B} F^*_{uv}(x)}^2} \\
        &= \sum_{x \in V} \lrbrace*{\sum_{u,v \in B} \Expect\lrbracket*{F^*_{uv}(x)^2} + 2\sum_{{u,v}\neq{u',v'} \in B} \Expect\lrbracket*{F^*_{uv}(x)} \Expect\lrbracket*{F^*_{u'v'}(x)}}.
    \end{align*}
    Since $F^*_{uv}(x)^2 = F^*_{uv}(x) \in  \set{0,1}$,
    \begin{align*}
        \Expect\lrbracket*{\con_2(F^*)^2} 
        &\leq \sum_{x \in V} \sum_{u,v \in B} \Expect\lrbracket*{F^*_{uv}(x)} + \sum_{x \in V} \paren*{\sum_{u,v \in B} \Expect\lrbracket*{F^*_{uv}(x)}}^2 \\
        &\leq \con_1(F) + \con_2(F)^2.
    \end{align*}
    Thus, there exists a boundary flow $F^*$ such that 
    $\con_2(F^*) \leq \sqrt{\con_2(F)^2 + \con_1(F)} \leq \con_2(F) + \sqrt{\con_1(F)}$.
    
    We divide the vertices of graph $G$ into two parts according to their degree of congestion. 
    Define $T = \set{v \in V : C_F(v) = \sum_{p \in \cP : v \in p} F(p) \geq 1}$, $\con_1(F, T) = \sum_{v \in T} C_F(v)$, and $\con_1(F, \widehat{T}) = \sum_{v \in V \setminus T} C_F(v)$. 
    Then we have
    \begin{align*}
        \sqrt{\con_1(F)} &= \sqrt{\con_1(F,T) + \con_1(F, \widehat{T})} \leq \sqrt{\con_1(F, T)} + \sqrt{\con_1(F, \widehat{T})} \\
        &\leq \con_2(F) + \sqrt{\card{V} - \card{T}} \leq \con_2(F) + \card{B}^{\frac{2}{1 + 4\epsilon}}.
    \end{align*}
    The first inequality follows by the inequality $\sqrt{x+y} \leq \sqrt{x} + \sqrt{y}$, and the last inequality holds since $\card{B} \geq \card{V}^{\frac{1}{4}+\epsilon}$.
\end{proof}

\begin{definition}[Semi-metric measure]
    Let $X$ be a nonempty set. 
    If the function $d:X \times X \rightarrow \mathbb{R}$ satisfies the following conditions for any $x,y,z \in X$,
    \begin{enumerate}
        % \begin{split}
        \item $d(x, y) \geq 0$;
        \item $d(x, y)=d(y, x)$;
        \item $d(x, y) \leq d(x, z)+d(z, y)$,
        % \end{split}&
    \end{enumerate}
    then the function $d$ is said to be a semi-metric measure on $X$. 
    The set $X$ endowed with the semi-metric $d$ is called a semi-metric space denoted by $(X,d)$.
\end{definition}

\begin{definition}[$\Lambda_s(G)$]
    A nonnegative vertex weight $s : V\longrightarrow \mathbb{R}_+$ induces a semi-metric $d_s :V \times V \longrightarrow \mathbb{R}_+$, where $d_s(u, v) = min_{p \in \mathcal{P}_{uv}} \sum_{x \in p} s(x)$.
    We define
    \begin{align*}
        \Lambda_s(G) = \frac{\sum_{u,v \in B} d_s(u, v)}{\sqrt{\sum_{v\in  V} s(v)^2}}.
    \end{align*}
\end{definition}

Now let us prove the following theorem serving for the proof of Theorem \ref{thm:lambda2_genus}.

\begin{theorem}\label{thm:con2-Lambda-dual}
    Let $G=(V,E)$ be a graph with boundary $B$, then
    \begin{align*}
        \underset{F}{\min} \con_2(F) = \underset{s : v \longrightarrow \RR_+}{\max} \Lambda_s(G),
    \end{align*}
    where the minimum is taken over all unit $K_{\card{B}}$-boundary flows in $G$, and the maximum is taken over all nonnegative
    weight functions on $V$.
\end{theorem}

\begin{proof} 
    The proof is similar to the proof of Theorem 2.2 in \cite{biswal2010eigenvalue}. 
    We need to modify is to let $P \in \set{0,1}^{\cP \times V}$ be the path-vertex incidence matrix and to let $Q \in \set{0,1}^{\cP \times \binom{B}{2}}$ be the path-endpair incidence matrix. 
    Here
    \begin{align*}
        P_{p, x} = 
        \begin{cases}
            1, & x \in p, \\
            0, & \textotherwise
        \end{cases}
        \quad 
        Q_{p, u v} = 
        \begin{cases}
            1, & p \in \cP_{uv}, \\
            0, & \textotherwise
        \end{cases}
    \end{align*}
    for $x \in V$ and $u,v \in B$. 
    Then we write $\max_{s: V \rightarrow \RR_{+}} \Lambda_{s}(G)$ as a convex program ($\mathrm{P}$) in standard form, with variables $(d, s) \in \Omega = \RR_{+}^{\binom{B}{2}} \times \RR_{+}^{V}$.
    \begin{align}
        \operatorname{minimize } & -\bm{1}^{\top} d & \nonumber \\
        \text { subject to } & Q d - P s \preceq 0 & \norm{s}_{2}^{2} -1\leq 0 \nonumber \\
        & s \succeq 0 & d \succeq 0.
        \tag{P}
    \end{align}
    Next, we introduce the Lagrangian multipliers $f \in \RR_+^{\cP}$ and $\mu \in \RR_+$ and write the Lagrangian function
    \begin{align*}
        L = (d,s,f,\mu) &= -\bm{1}^{\top}d + {f}^{\top}(Qd - Ps) + \mu({s}^{\top}s - \bm{1}) \\
        &={d}^{\top}({Q}^{\top} f - \bm{1}) + (\mu {s}^{\top} s-{f}^{\bm{T}} P s) - \mu.
    \end{align*}
    Therefore, the Lagrange dual $g(f, \mu) = \inf_{(d,s) \in \Omega} L(d, s, f, \mu)$ is given by
    \begin{align*}
        g(f,\mu) = \underset{d \succeq 0}{\inf}{d}^{\top}({Q}^{\top} f - 1) + \underset{s \succeq 0}{\inf}(\mu {s}^{\top} s - {f}^{\top} P s) - \mu.
    \end{align*}
    The dual program is then $\sup_{f,\mu} g(f,\mu)$. 
    In order to write it in a more tractable form, first observe
    that $g(f,\mu) = -\infty$ if ${Q}^{\top} f \prec 1$. 
    If we require ${Q}^{\top} f \succeq 1$, it is easy to see that the optimum must be attained when equality holds. 
    To minimize the quadratic part, set $\bigtriangledown (\mu {s}^{\top} s - {f}^{\top} P s) = 0$ to get $s = {P}^{\top} f/2\mu$. 
    With these substitutions, the dual objective simplifies to
    \begin{align*}
        g(f, \mu) = -\frac{\norm{P^{\top} f}_{2}^{2}}{4 \mu} - \mu.
    \end{align*}
    To maximize this quantity, set $\mu^{*} = \frac{1}{2} \norm{P^{\top} f}_{2}$, and we get $g(f, \mu) = -\norm{P^{\top} f}_{2}$. 
    Therefore, the final dual program $(\mathrm{P}^{*})$ is
    \begin{align}
        \text{minimize } & \norm{P^{\top} f}_{2} \nonumber \\
        \text{ subject to } & f \succeq 0 \quad Q^{\top} f = 1.
        \tag{P*}
    \end{align}
    When $P$ and $Q$ correspond to a $K_{\card{B}}$ demand graph for $G$, the dual optimum is precisely $\min_{f} \con_{2}(f)$, where the minimum is taken over all unit $K_{|B|}$-boundary flows.
    
    The theorem now follows from Slater's condition in convex optimization~\cite{boyd2004convex}. 
    According to Slater's condition for strong duality, if the feasible region for ($\mathrm{P}$) has nonempty interior, the values of ($\mathrm{P}$) and $(\mathrm{P}^{*})$ are equal.
\end{proof}

Next, we prove a lower bound of the 2-congestion. 
We need the following bound of planar crossing number.

\begin{theorem}[{Ajtai, Chv{\'a}tal, Newborn, and Szemer{\'e}di~\cite{ajtai1982crossing}}]\label{thm:planar-graph-crossing-bound}
    If $m \geq 4n$ then every planar drawing of a graph with $n$ vertices and $m$ edges contains at least $m^{3} / 100 n^{2}$  crossings.
\end{theorem}

\begin{corollary}\label{coro:planar-complete-ve-crossing-bound}
    The planar crossing number of $K_n\ (n \geq 9)$ has a lower bound
    \begin{align*}
        {\crossing}(K_{n}) \geq \frac{n^4}{2000}.   
    \end{align*}
\end{corollary}

\begin{proof}
    When $n \geq 9$, $m \geq 4n$, $n-1 \geq \frac{4n}{5}$, by Theorem \ref{thm:planar-graph-crossing-bound}, we have $\crossing(K_{n}) \geq \frac{n^4}{2000}$.
\end{proof}

\begin{theorem}[A lower bound of ${\con}_{2}(F)$]\label{thm:con2-bound}
    Let $G=(V, E)$ be a graph of genus $g$ with boundary $B$.
    Let $F$ be a unit $K_{\card{B}}$-boundary flow. 
    Then there exists a universal constant $c > 0$ such that $\con_{2}(F) \geq \frac{c \card{B}^{2}}{\sqrt{g}}\geq \frac{c \card{B}^{2}}{\sqrt{g+1}}$ for $\card{B} \geq \max\set{3 \sqrt{g}, \card{V}^{\frac{1}{4} + \epsilon}, 9}$ with $\epsilon > 0$.
\end{theorem}

\begin{proof}
    By Lemma \ref{lem:integral-flow} it suffices to prove the theorem when $F$ is an integral boundary flow. 
    
    By the definition of $F$, every pair of boundary vertices of $G$ serve as endpoints, and there is exactly one path $p$ connecting them such that $F(p) = 1$. 
    Draw $G$ on a surface $S$ with genus $g$. 
    Each edge of $K_{|B|}$, whose vertices consist of boundary vertices, corresponds to a unique path such that $F(p) = 1$. 
    The drawing of each edge of $K_{|B|}$ follows the path in $G$, and such a drawing induces a drawing of $K_{|B|}$. 
    In addition, if two edges of $K_{|B|}$ cross, then the position of crossing must be near some vertices of $G$, which belong to both paths corresponding to two edges. 
    See~\cref{fig:crossing_near_vertex}.
    
    \begin{figure}[htbp]
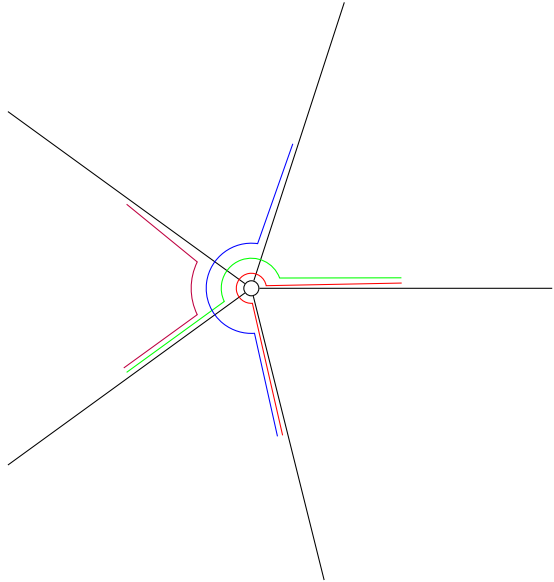

        \centering
        \includestandalone{crossing_near_vertex}
        \caption{Crossing of paths near a vertex in $G$}
        \label{fig:crossing_near_vertex}
    \end{figure}
    
    The first case is $g = 0$. 
    Suppose that there exists an integral $K_{\card{B}}$-boundary flow  $F$ with $\con_{2}(F) < \frac{\card{B}^{2}}{20\sqrt{5}}$. 
    The crossing number has an upper bound 
    \begin{align*}
        \sum_{v \in V} C_{F}(v)^{2} = \con_{2}(F)^{2} < \frac{\card{B}^{4}}{2000}.
    \end{align*}
    However, when $\card{B} \geq 9$, any drawing of $K_{\card{B}}$ in the  plane requires at least $\frac{\card{B}^{4}}{2000}$ crossings by~\cref{coro:planar-complete-ve-crossing-bound}, yielding a contradiction.
    
   The second case is $g >0$. Suppose there exists an integral $K_{\card{B}}$-boundary flow $F$ with  $\con_{2}(F) < \frac{\card{B}^{2}}{20\sqrt{5g}}$.
    The crossing number is bounded above by  
    \begin{align*}
        \sum_{v \in V} C_{F}(v)^{2} = \con_{2}(F)^{2} < \frac{\card{B}^{4}}{2000g}.
    \end{align*}
    Actually, when $\card{B} \geq 3\sqrt{g}$, any drawing of $K_{\card{B}}$ in a surface of genus $g$ requires at least $\frac{\card{B}^{4}}{64g}$ edge crossings (\cite[Theorem 3.1]{biswal2010eigenvalue} and~\cite[Theorem 2]{shahrokhi1996crossing}), yielding a contradiction.
    
    Therefore, $\con_{2}(F) \geq \frac{\card{B}^{2}}{20\sqrt{5g}}$ when $\card{B} \geq \max\set{3 \sqrt{g},|V|^{\frac{1}{4} + \epsilon},9}$.
\end{proof}

Let $(X,d)$ be a finite metric space. 
We recall the standard definitions of the padded decomposition and the modulus of padded decomposition.

\begin{definition}[{Padded decomposition}~\cite{krauthgamer2004measured,lee2010genus}] 
    Let $(X, d)$ be a finite metric space.  
    If $P$ is a partition of $X$, we will also regard it as a function $P: X \rightarrow 2^{X}$ such that for every $x \in X$, $P(x)$ is the unique part $C \in P$ for which $x \in C$.
    Let $\mu$ be a distribution over partitions of $X$, and $P$ a random partition distributed according to $\mu$. 
    We say that $\mu$ is $\kappa$-bounded if it holds that for every $C \in P$, we have $\operatorname{diam}(C) \leq \kappa$.
    We say that $\mu$ is $(\alpha, \delta)$-padded if
    \begin{align*}
        \operatorname{Pr}[B(x, \kappa/\alpha) \subseteq P(x)] \geq \delta
    \end{align*}
    for every $x \in X$. 
    Here $B(x, r) \coloneqq \set{y \in X : d(x, y) \leq r}$. 
\end{definition}

\begin{definition}[{The modulus of padded decomposibility~\cite{krauthgamer2004measured,lee2010genus}}]
  The modulus $\alpha(X, d; \delta, \kappa)$ of padded decomposibility of a finite metric space $(X, d)$ is defined as
    \begin{align*}
        \alpha(X, d; \delta, \kappa) &= \inf \set{\alpha: %X \text{ admits a } %\kappa \text{-bounded } \alpha \text{-padded partition } \mu 
        (X, d) \text{ admits a } \kappa\text{-bounded } (\alpha, \delta)\text{-padded random partition}}. \\
        \alpha(X, d; \delta) &= \inf_{\kappa > 0} \alpha(X, d; \delta, \kappa).
    \end{align*}
\end{definition}

Now we state two important theorems in \cite{biswal2010eigenvalue} and \cite{lee2010genus}.

\begin{theorem}[{\cite{biswal2010eigenvalue}}]\label{thm:minor-free-alpha-bound}
    Let  $G = (V, E)$ be a $K_{h}$-minor free, and $s : V \rightarrow \RR_{+}$ is a nonnegative weight function on vertices. 
    Then the induced semi-metric $d_{s}$ satisfies $\alpha\left(V, d_{s};\frac{1}{2}\right) = \bigO(h^{2})$.
    In particular, if $G$ is of orientable genus $g > 0$, then $\alpha\paren*{V, d_s;\frac{1}{2}} = \bigO(g)$.
\end{theorem}

\begin{theorem}[{\cite[Theorem 4.1]{lee2010genus}}]\label{thm:genus-kappa-alpha-bound}
    Let  $G = (V, E)$ be a graph with orientable genus $g > 0$, and $s : V \rightarrow \RR_{+}$ is a nonnegative weight function on vertices.
    Then for every  $\kappa>0$, $\alpha\left(V, d_{s} ; \frac{1}{8}, \kappa\right)=\bigO_{\kappa}(\log g)$.
\end{theorem}

Next, we give a theorem which is similar to~\cite[Theorem 4.4]{biswal2010eigenvalue}. 
For a semi-metric space $(X, d_s)$, a map $f: X \rightarrow \RR$ is said to be non-expansive if $|f(x)-f(y)| \leq d_s(x, y)$ for all $x, y \in X$.

\begin{theorem}\label{thm:non-expansive_map}
    For every $p \geq 1$, there exists a constant $C_{p} \geq 1$ such that for any semi-metric space $(V, d_s)$, there exists a non-expansive map  $f: V \rightarrow \mathbb{R}$ such that
    \begin{align*}
        \sum_{u, v \in B} d_s(u, v)^{p} \leq C_{p}\lrbracket*{\alpha(X, d_s;\delta)}^{p} \sum_{u, v \in B } |f(u)-f(v)|^{p}.
    \end{align*}
\end{theorem}

\begin{proof}
    Let $\kappa_{p} = \lrbracket*{\frac{1}{\card{B}^2} \sum_{u,v \in B}d_{s}(u, v)^{p}}^{\frac{1}{p}}$. 
    
    We divide the following argument into two cases.
    First we consider the case that many boundary vertices are clustered about a single vertex $x_{0} \in V$.
    \begin{enumerate}[label={(\Roman*)}]
        \item There exists $x_{0} \in V$ such that $
        \card*{S'} \geq \frac{\card{B}}{10}$, where $S' \coloneqq B(x_{0}, \frac{1}{4} \kappa_{p}) \cap B$ is the set of boundary vertices in $B(x_{0}, \frac{1}{4} \kappa_{p})$. 
        
        Let $S$ be the set of all the vertices in $B(x_{0}, \frac{1}{4} \kappa_{p})$. 
        We claim $f(u) = d_s(u, S)$ is the desired function. 
        We have
        \begin{align*}
            |B|^{2} \kappa_{p}^{p} &= \sum_{u, v \in B } d_s(u, v)^{p} \\
            &\leq 2^{p-1} \sum_{u, v\in B} \lrbracket*{d_s(u, x_{0})^{p} + d_s(v, x_{0})^{p}} \\
            &= 2^{p}\card{B} \sum_{u \in B} d_s(u, x_{0})^{p} \\
            &\leq 2^{p} \card{B} \sum_{u \in B} \lrbracket*{d_s(u, S)+\frac{\kappa_{p}}{4}}^{p} \\
            &\leq \frac{\card{B}^{2} \kappa_{p}^{p}}{2} + 2^{2p-1} \card{B} \sum_{u \in B} d_s(u, S)^{p}.
        \end{align*}
        
        The first and the last inequalities follow from the inequality $(x+y)^p \leq 2^{p-1} (x^p+y^p)$. 
        Hence, $\sum_{u \in B} d_s(u, S)^{p} \geq \frac{\card{B} \cdot \kappa_{p}^{p}}{2^{2p}}$. 
        Consider a path with boundary endpoints $u$ and $v$ which is one of the longest paths with boundary vertices as endpoints. 
        If $S$ contains all the boundary vertices, namely $S' = B$, then
        \begin{align*}
            d_s(u,v)^{p} \leq \paren*{d_s(u,x_{0}) + d_s(x_{0},v)}^p \leq \paren*{2 \cdot \frac{\kappa_p}{4}}^p = \frac{\kappa_{p}^p}{2^{p}}.
        \end{align*}
        According to the definition of $\kappa_{p}$, we have $d_s(u,v)^{p}\geq \kappa_{p}^p$, which leads to a contradiction.
        Therefore, there exist some boundary vertices that are not in $S$. We conclude that
        \begin{align*}
            \sum_{u, v \in B }(f(u)-f(v))^{p} &= \sum_{u, v \in B} \lrbracket*{d_s(u, S) - d_s(v, S)}^{p} \\
            &\geq \sum_{\substack{u,v \in B\\ u \notin S', v \in S'}} \lrbracket*{d_s(u, S) - d_s(v, S)}^{p} \\
            &=\sum_{\substack{u,v \in B\\ u \notin S', v \in S'}} d_s(u, S)^{p} \\
            &= \card{S'} \sum_{u \in B} d_s(u, S)^{p} \\
            &\geq \card{S'} \frac{ \card{B} \kappa_{p}^{p}}{2^{2p}}\\
            &\geq \frac{\card{B}}{10} \frac{ \card{B} \kappa_{p}^{p}}{2^{2p}}\\
            &= \frac{1}{10 \cdot 2^{2p}} \card{B}^{2} \kappa_{p}^{p} \\
            &= \frac{1}{10 \cdot 2^{2p}} \sum_{u, v \in B} d_s(u, v)^{p}.
        \end{align*}
        This finishes the clustered case.
        
        \item For every vertex $x$ in $V$, it holds $\card{S'} < \frac{\card{B}}{10}$, where $S' = B(x, \frac{1}{4} \kappa_{p}) \cap B$.
        
        Consequently $\card{V' \cap B} < \frac{\card{B}}{10}$ for any subset $V' \subseteq V$ with $\operatorname{diam}(V') \leq \frac{1}{4} \kappa_{p}$. 
        
        Let $P$ be a random partition of $V$ which is $\frac{\kappa_p}{4}$-bounded and $(\alpha,\delta)$-padded, where $\alpha = \alpha(V, d_s; \delta)$. 
        By the definition of $P$, for every $x \in X$ we have
        \begin{align*}
            \operatorname{Pr} \lrbracket*{B(x, \kappa_{p} /(4 \alpha)) \subseteq P(x)} \geq \delta.
        \end{align*}

        % Consider the random variable
        % \begin{align*}
        %     Y 
        %     &= \card{x \in B : B(x, \kappa_p / (4\alpha)) \subseteq P(x)} \\
        %     &= \sum_{x \in B} I\lrbracket*{B(x, \kappa_p / (4\alpha)) \subseteq P(x)},
        % \end{align*}
        % where $I[X]$ is the indicator random variable of $X$. 
        % By Markov inequality $\operatorname{Pr}[X \geq t] \leq \frac{\Expect[X]}{t}$, we get
        % \begin{align*}
        %     \Expect[Y] 
        %     &\geq \frac{\card{B}}{2} \operatorname{Pr}\lrbracket*{Y \geq \frac{\card{B}}{2}} 
        % \end{align*}
        Define 
        \begin{align*}
            B_{i} \coloneqq \set{x \in B : B(x, \kappa_{p} /(4 \alpha)) \subseteq P_i(x)}
        \end{align*}
        and $a_i = \operatorname{Pr}(P = P_i)$. 
        We have 
        \begin{align*}
            \sum_{x\in B} \operatorname{Pr}[B(x, \kappa_{p} /(4 \alpha)) \subseteq P(x)] & =\sum_{x\in B}\sum_ia_i\mathbbm{1}_{\{B(x, \kappa_{p} /(4 \alpha))\subseteq P_i(x)\}}\\
            &=\sum_ia_i\sum_{x\in B}\mathbbm{1}_{\{B(x, \kappa_{p} /(4 \alpha))\subseteq P_i(x)\}}\\
            &=\sum_ia_i|B_i|\geq\delta\card{B}.
        \end{align*}
        
        Therefore, there exists a partition $P_0$ and a subset $B_{0} \subseteq V$ such that 
        \begin{align*}
            B_{0} \coloneqq \set{x \in B : B(x, \kappa_{p} /(4 \alpha)) \subseteq P_0(x)}
        \end{align*}
        satisfies $\card{B_{0}} \geq \delta\card{B}$.
        Fix this choice of $P_{0}$ and $B_{0}$.
        
        Let $\set{\sigma_{C}}_{C \in P_{0}}$ be a collection of i.i.d. uniform $0 / 1$ random variables, one for each cluster $C \in P_{0}$ and define $S = \bigcup_{C \in P_{0} : \sigma_{C} = 0} C$. 
        Consider the random function $f: X \rightarrow \RR$ defined by $f(u) = d_s(u, S)$.
        
        Note that $f$ is a random function with repect to $\set{\sigma_{C}}$. 
        We now argue that
        \begin{align*}
            \Expect\lrbracket*{\sum_{u, v \in B} (f(u)-f(v))^{p}} \gtrsim \frac{\card{B}^2 \kappa_{p}^p}{\alpha^p} \gtrsim \alpha^{-p} \sum_{u, v \in B} d_s(u, v)^{p},    
        \end{align*}
        which implies that there exists a choice of $f: X \rightarrow \RR$ for which the sum is at least $\Omega(\alpha^{-p}) \sum_{u, v \in B} d_s(u, v)^{p}$.
        
        Note that by the definition of $P_{0}$, we have $\operatorname{diam}(C) \leq \kappa_{p} / 4$ for every $C \in P_{0}$, so the number of boundary vertices in $C$ cannot exceed $\frac{|B|}{10}$. 
        Therefore,
        \begin{align}\label{eq:triple_sum}
            \sum_{u, v \in B} (f(u)-f(v))^{p} &= \sum_{u, v \in B} \paren*{d_s(u, S) - d_s(v, S)}^{p} \notag \\
            &\geq \sum_{C \in P_{0}} \sum_{u \in C \cap B_{0}} \sum_{v \notin C} \paren*{d_s(u, S) - d_s(v, S)}^{p}.
        \end{align}
        Let us estimate  $\Expect\lrbracket*{(d_s(u, S) - d_s(v, S))^{p}}$ for $u \in C \cap B_{0}$ and $v \notin C$. 
        If $\sigma_{C} = 1$ and $\sigma_{P_{0}(v)} = 0$, then $v \in S$ and $u \notin S$. 
        Moreover, $d_s(u, S) \geq \frac{\kappa_{p}}{4 \alpha}$. 
        Otherwise there exists a vertex $x$ in $S$ such that $d_s(u, x) < \frac{\kappa_{p}}{4 \alpha}$, then $x \in C$, and by the definition of padded decomposition there is only one $C$ that contains $x$. 
        Hence $\sigma_{P_{0}(x)} = 1$ which contradicts with $x \in S$.
        
        Since $\sigma_{C}$ and $\sigma_{P_{0}(v)}$ are independent, by total expectation formula, we have
        \begin{align*}
            \Expect\lrbracket*{(d_s(u, S) - d_s(v, S))^{p}} &\geq \frac{1}{4} \Expect\lrbracket*{(d_s(u, S) - d_s(v, S))^{p} \mid \sigma_{C} = 1, \sigma_{P_{0}(v) = 0}} \\
            &\geq \frac{1}{4} \paren*{\frac{\kappa_{p}}{4 \alpha}}^{p}.
        \end{align*}
        Plug this into~\cref{eq:triple_sum} and use $\card{C\cap B} < \card{B} / 10$  for every $C \in P_{0}$, then it yields
        \begin{align*}
            \Expect\lrbracket*{\sum_{u, v \in B} (f(u) - f(v))^{p}} 
            &\geq \sum_{C \in P_{0}} \sum_{u \in C \cap B_{0}} \card{B \backslash C} \cdot \frac{\kappa_{p}^{p}}{2^{2p+2} \alpha^{p}} \\
            &\geq \sum_{C \in P_{0}} \sum_{u \in C \cap B_{0}} \frac{9 \card{B}}{10} \cdot \frac{\kappa_{p}^{p}}{2^{2p+2} \alpha^{p}} \\
            &= \card{B_{0}} \frac{9 \card{B}}{10} \frac{\kappa_{p}^{p}}{2^{2p+2} \alpha^{p}} \\
            &\geq \delta\card{B} \frac{9\card{B}}{10} \frac{\kappa_{p}^{p}}{2^{2p+2} \alpha^{p}} \\
            &\gtrsim \frac{\card{B}^2 \kappa_{p}^p}{\alpha^p}.
        \end{align*}
    \end{enumerate}
    The proof is complete. 
\end{proof}

Let $G = (V, E)$ be a graph with boundary $B$. 
Recall that $\sigma_{2}(G, B)$ is the first non-trivial eigenvalue of the Steklov eigenvalue of $G$ with boundary $B$. 
Recall the variational characterization of Steklov eigenvalues, we have
\begin{align}\label{eq:variation}
    \sigma_{2}(G, B) &= \min_{\substack{g: V \rightarrow \RR \\ \sum_{x \in B} g(x) = 0}} \frac{\sum_{u v \in E} |g(u) - g(v)|^{2}}{\sum_{u \in B} |g(u)|^{2}} \notag
    = \min_{f: V \rightarrow \RR} \frac{\sum_{u v \in E} |f(u)-f(v)|^{2}}{\sum_{u \in B} |f(u)-\bar{f}|^{2}} \notag \\
    &= 2\card{B} \cdot \min _{f: V \rightarrow \RR} \frac{\sum_{u v \in E} |f(u)-f(v)|^{2}}{\sum_{u, v \in B}|f(u)-f(v)|^{2}},
\end{align}
where $\bar{f} = \frac{1}{n} \sum_{x \in B} f(x)$ and $g(u) = f(u)-\bar{f}$.

\begin{theorem}\label{thm:Delta-B-alpha-Lambda-bound}
    Let $G=(V, E)$ be a graph with maximum degree $\Delta$ and boundary $B$. 
    Let $s: V \rightarrow \RR_{+}$ be a nonnegative weight function on vertices, and let $d_{s}$ be the induced semi-metric. 
    Then
    \begin{align*}
        \sigma_{2}(G, B) \lesssim \frac{\Delta \card{B}^{3} \lrbracket*{\alpha(V, d_s;\delta)}^{2}}{\Lambda_{s}(G)^{2}}.
    \end{align*}
\end{theorem}

\begin{proof}
    By~\cref{thm:non-expansive_map}, there exists a non-expansive map $f: V \to \RR_+$ such that 
    \begin{align*}
        \sum_{u, v \in B} d_s(u, v)^{p} \leq C_{p}\lrbracket*{\alpha(V, d_s;\delta)}^{p} \sum_{u, v \in B } |f(u)-f(v)|^{p}.
    \end{align*}
    For $uv \in E$, we have $d_{s}(u, v)^{2} = (s(u) + s(v))^2 \leq 2(s(u)^2 + s(v)^2)$. 
    By~\cref{eq:variation}, we get
    \begin{align*}
        \sigma_{2} &\lesssim \card{B} \lrbracket*{\alpha(V, d_s;\delta)}^{2} \frac{\sum_{u v \in E} d_{s}(u, v)^{2}}{\sum_{u, v \in B} d_{s}(u, v)^{2}} \\
        &\leq \card{B} \lrbracket*{\alpha(V, d_{s};\delta)}^{2} \frac{4 \Delta \sum_{v \in V} s(v)^{2}}{\sum_{u, v \in B} d_{s}(u, v)^{2}} \\
        &\leq 4 \Delta \card{B}^{3} \lrbracket*{\alpha(V, d_{s};\delta)}^{2} \frac{\sum_{v \in V} s(v)^{2}}{\lrbracket*{\sum_{u, v \in B} d_{s}(u, v)}^{2}} \\
        &= \frac{4 \Delta \card{B}^{3} \lrbracket*{\alpha(V, d_{s};\delta)}^{2}}{\Lambda_{s}(G)^{2}},
    \end{align*}
    where the last inequality follows from Cauchy-Schwarz inequality.
\end{proof}

%Now we can give the proof of~\cref{thm:lambda2_genus}.
%\begin{proof}[Proof of~\cref{thm:lambda2_genus}]
   %For any weight function $s: V \rightarrow \RR_{+}$, we have  $\alpha(V, d_{s}) = \bigO(\log g)$ by~\cref{thm:minor-free-alpha-bound}.
    %Since the planar graph does not contain a $K_5$ minor, $\alpha(V, d_{s}) = \bigO(1)$ for any planar graph by~\cref{thm:minor-free-alpha-bound}. 
    %In conclusion, we have $\alpha(V, d_{s}) = \bigO(\log(g+1))$.
    %Hence, by~\cref{thm:Delta-B-alpha-Lambda-bound,thm:con2-Lambda-dual,thm:con2-bound},
    %\begin{align*}
        %\sigma_{2} &\lesssim \frac{\Delta \card{B}^{3} \lrbracket*{\alpha(V, d_{s})}^{2}}{\Lambda_{s}(G)^{2}} \\
        %&\lesssim \frac{\Delta \card{B}^{3} (\log g)^2}{\Lambda_{s}(G)^{2}} \\
        %&\lesssim \frac{\Delta \card{B}^{3} (\log g)^2}{\frac{\card{B}^{4}} {g}} \\
        %&= \frac{\Delta g (\log g)^2}{\card{B}}. \qedhere
    %\end{align*}
%\end{proof}

Now we can give the proof of~\cref{thm:lambda2_genus}.
\begin{proof}[Proof of~\cref{thm:lambda2_genus}]
   For any weight function $s: V \rightarrow \RR_{+}$, we have  $\alpha(V, d_{s};\frac{1}{2}) = \bigO(g)$ by~\cref{thm:minor-free-alpha-bound}.
    Since the planar graph does not contain a $K_5$ minor, $\alpha(V, d_{s};\frac{1}{2}) = \bigO(1)$ for any planar graph by~\cref{thm:minor-free-alpha-bound}. 
    In conclusion, we have $\alpha(V, d_{s};\frac{1}{2}) = \bigO(g+1)$.
    Hence, by~\cref{thm:Delta-B-alpha-Lambda-bound,thm:con2-Lambda-dual,thm:con2-bound},
    \begin{align*}
        \sigma_{2} &\lesssim \frac{\Delta \card{B}^{3} \lrbracket*{\alpha(V, d_{s};\frac{1}{2})}^{2}}{\Lambda_{s}(G)^{2}} \\
        &\lesssim \frac{\Delta \card{B}^{3} (g+1)^2}{\Lambda_{s}(G)^{2}} \\
        &\lesssim \frac{\Delta \card{B}^{3} (g+1)^2}{\frac{\card{B}^{4}} {g+1}} \\
        &= \frac{\Delta (g+1)^3}{\card{B}}. \qedhere
    \end{align*}
\end{proof}

Using Theorem \ref{thm:genus-kappa-alpha-bound}, a similar result with $g>0$ based on $\kappa$, which is chosen based on the proof of Theorem \ref{thm:non-expansive_map} can be got as follow.
\begin{theorem}
    Let $G$ be a graph with (orientable) genus $g>0$  and boundary $B$ such that the vertex degree is bounded by $\Delta$.  
    If $\card{B} \geq \max\set*{3 \sqrt{g},|V|^{\frac{1}{4} + \epsilon}}$ for some $\epsilon > 0$, then
    \begin{align*}\label{eq:genus>0_lambda2}
        \sigma_2 \leq \bigO_\kappa\paren*{\frac{\Delta g(\log g)^2}{|B|}}.
    \end{align*}
\end{theorem}

\subsection{Proof of~\texorpdfstring{\cref{thm:lambda2_planar-crossing}}{Theorem}}
First, we recall a lemma as follow.
\begin{lemma}[{\cite[Corollary 2.3]{lin_first_2024} or~\cite[Theorem 4.2]{spielman_spectral_2007}}]\label{coro:kissing_cap}
    Let $G = (V,E)$ be a planar graph with $\card{V} = n$ vertices. 
    Then there exists a set of caps $\set{C_1, \ldots, C_n}$ in the sphere with disjoint interiors such that $C_i$ touches $C_j$ if and only if $(i,j) \in E$. 
    Moreover, suppose $\card{V} \geq 3$, and $B \subset V$ is a subset of the vertex set with $\card{B} \geq 2$. 
    Then we can further require that
    \begin{align*}
        \sum_{x \in B} m(C_x) = \bm{0}.
    \end{align*}
\end{lemma}

Now we give the Steklov version of the embedding lemma.

\begin{lemma}[{Embedding Lemma for the first Steklov eigenvalue~\cite[Lemma 2.4]{lin_first_2024}}]\label{lem:embedding_steklov}
    Let $G = (V,E)$ be a graph with $\card{V} = n$ vertices. 
    Let $B \subset V$ be the boundary of the graph and $\card{B} \geq 2$. 
    Then the second Steklov eigenvalue $\sigma_2$ of the pair $(G, B)$ is given by
    \begin{align*}
        \sigma_2 = \min \frac{\sum_{(x,y) \in E} \norm{v_x - v_y}^2}{\sum_{x \in B} \norm{v_x}^2},
    \end{align*}
    where the minimum is taken over the vectors $v_x \subset \RR^m$, $x \in V$ such that 
    \begin{align*}
        \sum_{x \in B} v_x = \bm{0}.
    \end{align*}
    Here $\bm{0}$ is the zero vector. 
\end{lemma}

We are prepared to prove \cref{thm:lambda2_planar-crossing}. 

\begin{proof}[{Proof of \cref{thm:lambda2_planar-crossing}}]
    Since the crossing number of $G$ is $X$, we can divide the edges into two parts $E_1$ and $E_2$ such that $\card{E_2} = X$ and $(V, E_1)$ is a planar graph.
    By \cref{coro:kissing_cap}, there is a representation of $(V,E_1)$ by kissing caps on the unit sphere so that the centroid of the centers of the caps corresponding to $\delta \Omega$ is the origin. 
    Let $v_x \in S^2, x \in V$ be the centers of the caps. 
    And we have $\sum_{x \in \delta \Omega} v_x = \bm{0}$. 
    
    Let $r_x, x \in V$ be the radii of the caps. 
    If the cap $x$ kisses the cap $y$, then the length of the edge from $v_x$ to $v_y$ is at most $(r_x + r_y)^2$, which is no larger than $2(r_x^2 + r_y^2)$. 
    Hence 
    \begin{align*}
        \sum_{(x,y) \in E_1} \norm{v_x - v_y}^2 \leq 2 \Delta \sum_{x \in V} r_x^2,
    \end{align*}
    and
    \begin{align*}
        \sum_{(x,y) \in E_2} \norm{v_x - v_y}^2 \leq 4 X,
    \end{align*}
    where $\Delta$ is the maximum degree of the graph $G$. 
    Since the caps do not overlap, we have
    \begin{align*}
        \sum_{x \in V} \pi r_x^2 \leq 4\pi.
    \end{align*} 
    Moreover, $\norm{v_x} = 1$ since the vectors are on the unit sphere. 
    Now we apply \cref{lem:embedding_steklov}, and we find that
    \begin{align*}
        \sigma_2 \leq \frac{8\Delta + 4X}{\card{B}},
    \end{align*}
    as required, the proof is complete.
\end{proof}

\subsection{Proofs of other upper bounds}

Let $L$ be the Laplacian matrix of the graph $G = (V, E)$. 
Let $B$ be the boundary of the graph $G$. 
For every positive real number $r$, we define a diagonal matrix $D^{(r)}$ by
\begin{align*}
    D^{(r)}(x,x) = 
    \begin{cases}
        1, & x \in B, \\
        r, & \textotherwise.
    \end{cases}
\end{align*}
Next we consider the eigenvalues of the matrix $L^{(r)} = D^{(r)} L D^{(r)}$.

\begin{lemma}[\cite{hassannezhad_higher_2020}]\label{lem:steklov_limit}
    Let $G = (V, E)$ be a graph with boundary $B$. 
    For every positive real number $r$, the matrix $L^{(r)}$ is defined as above.
    Let $\mu_1^{(r)} \leq \mu_2^{(r)} \leq \cdots \leq \mu_{|V|}^{(r)}$ be the eigenvalues of $L^{(r)}$. 
    Then
    \begin{align*}
        \lim_{r \to +\infty} \mu_k^{(r)} = \sigma_k, \quad 1 \leq k \leq \card{B},
    \end{align*}
    and
    \begin{align*}
        \lim_{r \to +\infty} \mu_k^{(r)} = +\infty, \quad \card{B} < k \leq \card{V}.
    \end{align*}
\end{lemma}

Next we invoke a technical lemma to adjust the eigenvalues. 

\begin{lemma}\label{lem:remove_first_two_eigenspace}
    Suppose the spectral decomposition of a positive semidefinite real matrix $A$ is given by
    \begin{align*}
        A = \sum_{i=1}^n \mu_i \xi_i \xi_i^\top,
    \end{align*}
    where $\mu_1 \leq \mu_2 \leq \cdots \mu_n$ are the eigenvalues of $A$ and $\xi_1, \xi_2, \ldots, \xi_n$ is an orthonormal basis of eigenvectors.
    Let $P = \xi_1 \xi_1^\top$ be the projection matrix onto the space spanned by $\xi_1$. 
    Then
    \begin{align*}
        A - \mu_1 P - \mu_2 (I - P)
    \end{align*}
    is positive semidefinite. 
\end{lemma}

\begin{proof}
    Note that
    \begin{align*}
        A - \mu_1 P - \mu_2 (I - P) &= \sum_{i=3}^n (\mu_i - \mu_2) \xi_i \xi_i^\top. \qedhere
    \end{align*}
\end{proof}

\begin{proof}[{Proof of \cref{thm:bound_min_deg}}]
    Take $A = L^{(r)}$ in~\cref{lem:remove_first_two_eigenspace}. 
    Note that $\mu_1 = 0$ and $\xi_1 = \eta / \norm{\eta}$, where the vector $\eta$ is given by
    \begin{align*}
        \eta(x) = 
        \begin{cases}
            1, & x \in B, \\
            1/r, & \textotherwise.
        \end{cases}
    \end{align*}
    Since $A - \mu_1 P - \mu_2 (I - P)$ is positive semidefnite, each of its diagonal entry must be nonnegative. 
    For every $u \in \delta\Omega$, we calculate the $(u,u)$-entry and get
    \begin{align*}
        \deg(u) - \mu_2^{(r)}\paren*{1 - \dfrac{1}{\card{B} + (\card{V} - \card{B})/r^2}} \geq 0. 
    \end{align*}
    Namely
    \begin{align*}
        \mu_2^{(r)} \leq \dfrac{\card{B} + (\card{V} - \card{B})/r^2}{\card{B} - 1 + (\card{V} - \card{B})/r^2} \deg(u).
    \end{align*}
    Let $r$ go to $+\infty$ and we get
    \begin{align*}
        \lambda_2 \leq \frac{\card{B}}{\card{B} - 1} \deg(u).
    \end{align*}
    Finally take $u$ as the vertex in $B$ with minimum degree. 
\end{proof}

We need the Cauchy interlacing Theorem to get relations among eigenvalues.

\begin{lemma}[{Cauchy interlacing Theorem~\cite{brouwer_spectra_2012}}]\label{thm:interlacing}
    Let $N$ be a principal submatrix of a real symmetric matrix $M$. 
    Then the eigenvalues of $N$ interlace the eigenvalues of $M$. 
\end{lemma}

We are prepared to prove~\cref{thm:interlace_bound}.

\begin{proof}[{Proof of \cref{thm:interlace_bound}}]
    Take $M = L^{(r)}$ in~\cref{thm:interlacing}. 
    Note that $N$ coincides the principal submatrix of $M$ indexed by $B$. 
    By~\cref{thm:interlacing}, the eigenvalues of $N$ interlace the eigenvalues of $M$. 
    Therefore
    \begin{align*}
        \mu_k(L^{(r)}) \leq \mu_k(N), \quad k = 1,2, \ldots, \card{B}.
    \end{align*}
    The conclusion follows by~\cref{lem:steklov_limit}.
\end{proof}

\section{Comparison}
In this section, we first give two more results on the upper bound of the Steklov eigenvalues.

\begin{theorem}\label{thm:bound_degree_diameter}
    Let $G$ be a graph with boundary $B$ such that the maximum degree $\Delta \geq 3$ and the boundary diameter $D_B \geq 2t+2 \geq 4$. 
    Then
    \begin{align*}
        \sigma_2 \leq \dfrac{(q+1) \left(q^{t+1}-q^t+1\right)}{q^{t+1}}=q-\frac{q^t-q-1}{q^{t+1}},
    \end{align*}
    where $q = \Delta - 1$. 
    The right-hand side of the inequality is monotone decreasing with respect to $t$. 
\end{theorem}
\begin{proof}[{Proof of \cref{thm:bound_degree_diameter}}]
    Let $x_0$ and $y_0$ be two boundary vertices that are at distance $D_B$ in $G$. 
    We partition the vertex set $V$ into $(2t+3)$ parts $L_0, L_1, \ldots, L_t, M, R_t, \ldots, R_1, R_0$ are follows. 
    \begin{align*}
        L_i &= \set{v \in V \mid \dist(x_0, v) = i}, & & 0 \leq i \leq t, \\
        R_i &= \set{v \in V \mid \dist(v, y_0) = i}, & & 0 \leq i \leq t, \\
        M &= V \setminus \paren*{\bigcup_{i=0}^t L_i \cup \bigcup_{i=0}^t R_i}.
    \end{align*}
    For convenience we set $L_{t+1} = R_{t+1} = M$. 
    Naturally it induce a partition on the edge set, namely,
    \begin{align*}
        E_{i}^L &= E(L_i, L_{i}), & & 0 \leq i \leq t, \\
        E_{i}^R &= E(R_i, R_{i}), & & 0 \leq i \leq t, \\
        E^M &= E(M, M), & &  \\
        E_{i,i+1}^L &= E(L_i, L_{i+1}), & & 0 \leq i \leq t, \\
        E_{i,i+1}^R &= E(R_i, R_{i+1}), & & 0 \leq i \leq t.
    \end{align*}
    To establish an upper bound for $\sigma_2$, we define a test function $f$ as follows. 
    \begin{align*}
        f(v) = 
        \begin{cases}
            a q^{t+1-i}, & v \in L_i,\ 0 \leq i \leq t, \\
            b q^{t+1-i}, & v \in R_i,\ 0 \leq i \leq t, \\
            0, & v \in M,
        \end{cases}
    \end{align*}
    where $q = \Delta - 1$ and $a, b \in \RR$ are chosen nonzero numbers such that $\sum_{v \in B} f(x) = 0$. 
    Then we calculate $R(f)$, namely
    \begin{align*}
        \sigma_2 &\leq R(f) \\
        &= \dfrac{\sum_{(u,v) \in E} ((f(u)-f(v))^2 }{\sum_{u \in B} f^2(u)} \\
        &\leq \dfrac{\sum_{(u,v) \in E^M \cup \bigcup_{i=0}^t (E_i^L \cup E_i^R \cup E_{i,i+1}^L \cup E_{i,i+1}^R)} ((f(u)-f(v))^2}{\sum_{u \in \set{x_0, y_0}} f^2(u)} \\
        &\leq \dfrac{(a^2 + b^2) \paren*{\sum_{i=0}^{t-1} (q+1) q^i \paren*{q^{t+1-i} - q^{t-i}}^2 + (q+1) q^t \paren*{q - 0}^2}}{(a^2 + b^2) \paren*{q^{t+1}}^2} \\
        &= \dfrac{(q+1) \left(q^{t+1}-q^t+1\right)}{q^{t+1}}.
    \end{align*}
    At last we prove the monotonicity. 
    Note that 
    \begin{align*}
        \dfrac{\dd}{\dd t} \dfrac{(q+1) \left(q^{t+1}-q^t+1\right)}{q^{t+1}} &= - \dfrac{(q+1) \ln(q)}{q^{t+1}} < 0. \qedhere
    \end{align*}
\end{proof}

\begin{theorem}\label{thm:degree_bound_general}
    Let $G = (V,E)$ be a graph with boundary $B$ such that the degree sequence of the boundary vertices is $d_1 \leq d_2 \leq \cdots \leq d_{\card{B}}$. 
    Let $d_1' \leq d_2' \leq \cdots \leq d_{\card{B}}'$ be the degree sequence of boundary vertices inside $B$, namely the degree sequence of $G' = (B, E(B,B))$. 
    Let $S_1 = \sum_{i=1}^{\card{B}} d_i$, $S_2 = \sum_{i=1}^{\card{B}} d_i^2$, and $S_1' = \sum_{i=1}^{\card{B}} d_i'$. 
    Then
    \begin{align*}
        \sigma_k &\leq \dfrac{1}{\card{B}} \lrbrace*{S_1 + \sqrt{\dfrac{(k-1)}{\card{B}-k+1} \Bigl[\card{B} (S_2 + S_1') - S_1^2\Bigr]}}, \\
        &\leq \dfrac{1}{\card{B}} \lrbrace*{S_1 + \sqrt{\dfrac{(k-1)}{\card{B}-k+1} \Bigl[\card{B} (S_2 + S_1) - S_1^2\Bigr]}},
    \end{align*}
    for $k = 1,2, \ldots, \card{B}$.
\end{theorem}
The proof of ~\cref{thm:degree_bound_general} uses techniques in~\cite{zhang_kth_2001}.

\begin{proof}[{Proof of~\cref{thm:degree_bound_general}}]
    Let $L$ be the Laplacian matrix of $G$, and let $N$ be the principal submatrix of $L$ indexed by $B$. 
    By~\cref{thm:interlace_bound}
    \begin{align*}
        \sigma_k \leq \mu_k(N),\quad k = 1,2, \ldots, \card{B}.
    \end{align*}
    Next we give an upper bound of $\mu_k(N)$ for $k = 1,2, \ldots, \card{B}$. 
    % If $k \geq 2$, then
    \begin{align*}
        \Tr(N^2) &= \sum_{i=1}^{k-1} \mu_i^2(N) + \sum_{i=k}^{\card{B}} \mu_i^2(N) \\
        &\geq \dfrac{\paren*{\sum_{i=1}^{k-1} \mu_i(N)}^2}{k-1} + \dfrac{\paren*{\sum_{i=k}^{\card{B}} \mu_i(N)}^2}{\card{B}-k+1}.
    \end{align*}
    Denote by $T$ the sum $\sum_{i=k}^{\card{B}} \mu_i(N)$. 
    Then
    \begin{align*}
        \Tr(N^2) \geq \dfrac{\paren*{S_1 - T}^2}{k-1} + \dfrac{T^2}{\card{B}-k+1}.
    \end{align*}
    Therefore
    \begin{align*}
        (\card{B}-k+1) \lambda_k \leq T \leq \dfrac{(\card{B}-k+1)S_1 + \sqrt{(k-1) (\card{B}-k+1) \Bigl[\card{B} \Tr(N^2) - S_1^2 \Bigr]}}{\card{B}}.
    \end{align*}
    Hence
    \begin{align*}
        \lambda_k &\leq \dfrac{1}{\card{B}} \lrbrace*{S_1 + \sqrt{\dfrac{(k-1)}{\card{B}-k+1} \Bigl[\card{B} \Tr(N^2) - S_1^2 \Bigr]}}.
    \end{align*}
    Note that
    \begin{align*}
        \Tr(N^2) = S_2 + S_1'.
    \end{align*}
    So
    \begin{align*}
        \lambda_k &\leq \dfrac{1}{\card{B}} \lrbrace*{S_1 + \sqrt{\dfrac{(k-1)}{\card{B}-k+1} \Bigl[\card{B} (S_2 + S_1') - S_1^2\Bigr]}}. \qedhere
    \end{align*}
    % If $k=1$, similar argument gives 
    % \begin{align}
    %     \lambda_1 &\leq \sqrt{\dfrac{S_2 + S_B}{\card{B}}}.
    % \end{align}
\end{proof}

\begin{proof}[{Proof of \cref{thm:degree_bound}}]
    We use~\cref{thm:interlace_bound}.
    Since the boundary vertices form an independent set of $G$, the matrix $N$ is a diagonal matrix with entries given by degree sequence of the boundary vertices. 
    Hence the eigenvalues of $N$ are
    \begin{align*}
        d_1 \leq d_2 \leq \cdots &\leq d_{\card{B}}. \qedhere
    \end{align*}
\end{proof}

%\cref{thm:bound_min_deg} is an extention of a theorem by Fiedler~\cite{fiedler_algebraic_1973}. 
%Suppose $\lambda_2 = \lambda_2(G)$ is the Fiedler eigenvalue of a graph $G$ and $\delta_V$ is the minimum vertex degree of $G$.  
%Then
%\begin{align}
%   \lambda_2 \leq \frac{\card{V}}{\card{V} - 1} \cdot \delta_{V}. 
%\end{align}
%There is another bound given by Fiedler~\cite{fiedler_algebraic_1973}, namely
%\begin{align}
%   \lambda_2 \leq \delta_{V}. 
%\end{align}
%This cannot be extended to $\sigma_2 \leq \delta_{B}$ (check \cref{empl:almost_K_2_D}). 

%For a planar graph $G$, there is the uniform bound~\cite[Theorem 3.2]{molitierno_algebraic_2006}
%\begin{align}
%   \lambda_2 \leq 4.
%\end{align}
%This cannot be extended to $\sigma_2 \leq c$ for a positive constant $c$ as well (check \cref{empl:almost_K_2_D}). 

%For a noncomplete graph $G = (V,E)$, there is the bound~\cite{belhaiza_variable_2005}
%\begin{align}
%   \lambda_2 \leq \floor{-1 + \sqrt{1 + 2 |E|}}.
%\end{align}
%This cannot be extended to $\sigma_2 = \bigO(|E|^{1/2})$ as well (check~\cref{empl:almost_K_2_D}).

%The comparsion between Laplacian eigenvalues and Steklov eigenvalues are summarized in~\cref{tab:compare}.

Then we provide an example to show that some bounds on Laplacian eigenvalues can not be extended to Steklov eigenvalues. 

\begin{example}\label{empl:almost_K_2_D}
    Consider the complete bipartite graph $K_{2,\Delta}$. 
    Choose a vertex in the part with $\Delta$ vertices, say $v_1$. 
    We remove an edge adjacent to $v_1$, and add an edge from $v_1$ to another vertex in the part with $D$ vertices, say $v_2$. 
    Then we obtain a graph which is almost the complete biparte graph $K_{2,\Delta}$, denoted by $K_{2,\Delta}^{\vee}$. 
    See \cref{fig:almost_K_2_D}. 
    Since $\sigma_2 = \Delta - \frac{4}{5} > \delta_B$, the inequality $\sigma_2 \leq \delta_B$ is disproved. 
    Since $K_{2, \Delta}^\vee$ is planar and $\Delta - \frac{4}{5}$ can be arbitrarily large, the inequality $\sigma_2 \leq c$ for a positive number $c$ is disproved. 
    Since $\card{E} = 2 \Delta$ and $\sigma_2 = \Delta - \frac{4}{5} = \Omega(\Delta)$, the bound $\sigma_2 = \bigO(\card{E}^{1/2})$ is disproved.
    \begin{figure}[htbp]
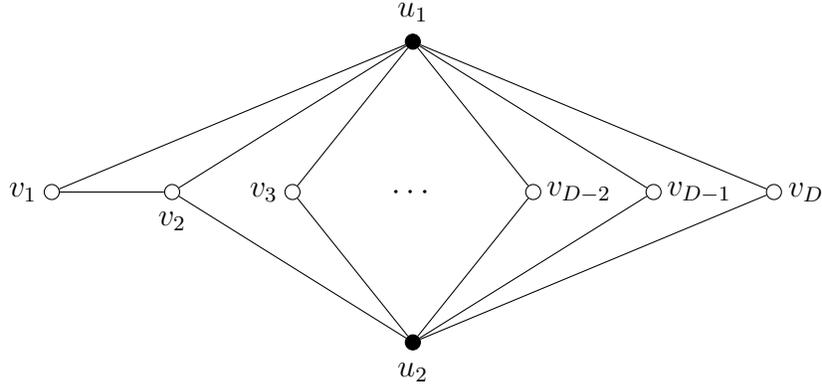

        \centering
        \includestandalone{almost_K_2_D}
        \caption{A graph with boundary $(G, B)$ of maximum degree $\Delta$ such that $\displaystyle \sigma_2 = \Delta - \frac{4}{5}$.}
        \label{fig:almost_K_2_D}
    \end{figure}
\end{example}

\begin{table}[htbp]
    \begin{tabular}{| c | c | c |} 
        \hline
        Graph Class & Laplacian eigenvalue & Steklov eigenvalue \\ 
        \hline\hline
        General & $\lambda_2 = \bigO\paren*{\frac{\Delta (g+1)^3}{\card{V}}}$~\cite{biswal2010eigenvalue}  & $\sigma_2 = \bigO\paren*{\frac{\Delta (g+1)^3}{\card{B}}}$~[\cref{thm:lambda2_genus}] \\[2ex]
        General & $\lambda_2 \leq \dfrac{8\Delta + 4X}{\card{V}}$~\cite{spielman_spectral_2007} & $\sigma_2 \leq \dfrac{8\Delta + 4X}{\card{B}}$~[\cref{thm:lambda2_planar-crossing}] \\ [2ex]
        % \hline
        General & $\lambda_2 \leq \frac{\card{V}}{\card{V} - 1} \cdot \delta_{V}$~\cite{fiedler_algebraic_1973} & $\sigma_2 \leq \frac{\card{B}}{\card{B} - 1} \cdot \delta_{B}$~[\cref{thm:bound_min_deg}] \\[2ex]
        % \hline
        General & $\lambda_2 \leq \delta_{V}$~\cite{fiedler_algebraic_1973} & \sout{$\sigma_2 \leq \delta_{B}$}~[\cref{empl:almost_K_2_D}] \\[2ex]
        % \hline
        Planar & $\lambda_2 \leq 4$~\cite{molitierno_algebraic_2006} & \sout{$\sigma_2 \leq c$}~[\cref{empl:almost_K_2_D}] \\[2ex]
        % \hline
        Noncomplete & $\lambda_2 \leq \floor{-1 + \sqrt{1 + 2 |E|}}$ & \sout{$\sigma_2 = \bigO(|E|^{1/2})$}~[\cref{empl:almost_K_2_D}] \\[2ex]
        & ~\cite{belhaiza_variable_2005} & \\[2ex]
        % \hline
        $\Delta \geq 3, D \geq 4$ & ~\cite[Lemma 1.14]{chung_spectral_1997} & $\sigma_2 \leq q-\dfrac{q^t-q-1}{q^{t+1}}$~[\cref{thm:bound_degree_diameter}] \\[2ex]
        \hline
        General & ~\cite{zhang_kth_2001} & $\sigma_k \leq \dfrac{1}{\card{B}} \lrbrace*{S_1 + \sqrt{\dfrac{(k-1)\Bigl[\card{B} (S_2 + S_1') - S_1^2\Bigr]}{\card{B}-k+1} }}$ \\[2ex]
        & & [\cref{thm:degree_bound_general}] \\[2ex]
        % \hline
        Special & $\lambda_k \leq d_k$ & $\sigma_k \leq d_k$~[\cref{thm:degree_bound}] \\[2ex]
        % independent set & & \\[2ex]
        % \hline
        \hline
    \end{tabular}
    \caption{Comparison of Laplacian eigenvalues and Steklov eigenvalues}
    \label{tab:compare}
\end{table}

\section{Some questions}
Our first result falls into the cubic growth in genus $g$, which is different from the linear growth in the case of the compact smooth orientable Riemann manifolds. 
This leads to a natural question.
\begin{question}
    What is the optimal growth of $g$ in the upper bound of the first Steklov eigenvalue for genus $g$ graph with boundary $B$ if the degree is bounded by $\Delta$?
\end{question}
It is well known as Robertson-Seymour Theorem \cite{ROBERTSON2004325} that a graph can be embedded into a surface of genus $g$ if and only if $G$ do not contain graphs in $\mathcal{F}_g$ as minors, where $\mathcal{F}_g$ is a finite set of graphs which only depends on $g$. 
A similar question is proposed.
\begin{question}
    What is the upper bound of the first Steklov eigenvalue for $K_h$-minor-free graphs with boundary $B$ if the degree is bounded by $\Delta$?
    How about the optimal growth of $h$ in the upper bound?
\end{question}

\section*{Acknowledgements}
Huiqiu Lin was supported by the National Natural Science Foundation of China (No. 12271162, No. 12326372), and Natural Science Foundation of Shanghai (No. 22ZR1416300 and No. 23JC1401500) and The Program for Professor of Special Appointment (Eastern Scholar) at Shanghai Institutions of Higher Learning (No. TP2022031).

\bibliographystyle{alpha}
\bibliography{main}

\end{document}

%% file: crossing_near_vertex.tex
\begin{tikzpicture}
    [scale=2, every node/.style={circle, inner sep=1pt, minimum size = 2mm}]
    \node[draw] (v0) at (0, 0) {};
    
    \coordinate (v1) at (0:2) {};
    \coordinate (v2) at (72:2) {};
    \coordinate (v3) at (144:2) {};
    \coordinate (v4) at (216:2) {};
    \coordinate (v5) at (284:2) {};

    \draw (v0) -- (v1);
    \draw (v0) -- (v2);
    \draw (v0) -- (v3);
    \draw (v0) -- (v4);
    \draw (v0) -- (v5);

    \draw[red] (2:1) -- (10:0.1) arc[start angle = 10, end angle = 274, radius = 0.1] -- (282:1);
    \draw[green] (4:1) -- (20:0.2) arc[start angle = 20, end angle = 206, radius = 0.2] -- (214:1);
    \draw[blue] (74:1) -- (82:0.3) arc[start angle = 82, end angle = 274, radius = 0.3] -- (280:1);
    \draw[purple] (146:1) -- (154:0.4) arc[start angle = 154, end angle = 206, radius = 0.4] -- (212:1);

\end{tikzpicture}

%% file: almost_K_2_D.tex
\begin{tikzpicture}
    [scale=2, every node/.style={circle, inner sep=1pt, minimum size = 2mm}]
    \node[draw, fill=black, label=above:{$u_1$}] (u1) at (0, 1) {};
    \node[draw, fill=black, label=below:{$u_2$}] (u2) at (0, -1) {};
    
    \node[draw, label=left:{$v_1$}] (v1) at (-2.4, 0) {};
    \node[draw, label=below:{$v_2$}] (v2) at (-1.6, 0) {};
    \node[draw, label=left:{$v_3$}] (v3) at (-0.8, 0) {};

    \node (v4) at (0,0) {$\cdots$};
    
    \node[draw, label=right:{$v_{D-2}$}] (v4) at (0.8, 0) {};
    \node[draw, label=right:{$v_{D-1}$}] (v5) at (1.6, 0) {};
    \node[draw, label=right:{$v_D$}] (v6) at (2.4, 0) {};

    \draw (u1) -- (v1);
    \draw (u1) -- (v2) -- (u2);
    \draw (u1) -- (v3) -- (u2);
    \draw (u1) -- (v4) -- (u2);
    \draw (u1) -- (v5) -- (u2);
    \draw (u1) -- (v6) -- (u2);
    \draw (v1) -- (v2);

\end{tikzpicture}

%% file: main.bbl
\begin{thebibliography}{BdAHO05}

\bibitem[ACNS82]{ajtai1982crossing}
Mikl{\'o}s Ajtai, Va{\v{s}}ek Chv{\'a}tal, Monroe~M Newborn, and Endre Szemer{\'e}di.
\newblock Crossing-free subgraphs.
\newblock In {\em North-Holland Mathematics Studies}, volume~60, pages 9--12. Elsevier, 1982.

\bibitem[BdAHO05]{belhaiza_variable_2005}
Slim Belhaiza, Nair Maria~Maia de~Abreu, Pierre Hansen, and Carla~Silva Oliveira.
\newblock Variable neighborhood search for extremal graphs. {XI}: {Bounds} on algebraic connectivity.
\newblock In {\em Graph theory and combinatorial optimization.}, pages 1--16. New York, NY: Springer, 2005.

\bibitem[BH12]{brouwer_spectra_2012}
Andries~E. Brouwer and Willem~H. Haemers.
\newblock {\em Spectra of graphs}.
\newblock Universitext. Berlin: Springer, 2012.
\newblock ISSN: 0172-5939.

\bibitem[BLR10]{biswal2010eigenvalue}
Punyashloka Biswal, James~R Lee, and Satish Rao.
\newblock Eigenvalue bounds, spectral partitioning, and metrical deformations via flows.
\newblock {\em Journal of the ACM}, 57(3):1--23, 2010.

\bibitem[Bro01]{brock_isoperimetric_2001}
F.~Brock.
\newblock An isoperimetric inequality for eigenvalues of the {Stekloff} problem.
\newblock {\em ZAMM. Zeitschrift für Angewandte Mathematik und Mechanik}, 81(1):69--71, 2001.

\bibitem[BV04]{boyd2004convex}
Stephen Boyd and Lieven Vandenberghe.
\newblock {\em Convex optimization}.
\newblock Cambridge university press, 2004.

\bibitem[Che24]{chen2024upper}
Hang Chen.
\newblock The upper bound of the harmonic mean of the steklov eigenvalues in curved spaces.
\newblock {\em Bulletin of the London Mathematical Society}, 56(3):931--944, 2024.

\bibitem[Chu97]{chung_spectral_1997}
Fan R.~K. Chung.
\newblock {\em Spectral graph theory}, volume~92 of {\em Reg. {Conf}. {Ser}. {Math}.}
\newblock Providence, RI: AMS, American Mathematical Society, 1997.
\newblock ISSN: 0160-7642.

\bibitem[Esc97]{escobar_geometry_1997}
José~F. Escobar.
\newblock The geometry of the first non-zero {Stekloff} eigenvalue.
\newblock {\em Journal of Functional Analysis}, 150(2):544--556, 1997.

\bibitem[Esc99]{escobar_isoperimetric_1999}
José~F. Escobar.
\newblock An isoperimetric inequality and the first {Steklov} eigenvalue.
\newblock {\em Journal of Functional Analysis}, 165(1):101--116, 1999.

\bibitem[Fie73]{fiedler_algebraic_1973}
Miroslav Fiedler.
\newblock Algebraic connectivity of graphs.
\newblock {\em Czechoslovak Mathematical Journal}, 23:298--305, 1973.

\bibitem[FS19]{fraser2019shape}
Ailana Fraser and Richard Schoen.
\newblock Shape optimization for the steklov problem in higher dimensions.
\newblock {\em Advances in Mathematics}, 348:146--162, 2019.

\bibitem[FW17]{fall2017profile}
Mouhamed~Moustapha Fall and Tobias Weth.
\newblock Profile expansion for the first nontrivial steklov eigenvalue in riemannian manifolds.
\newblock {\em Communications in Analysis and Geometry}, 25(2):431--463, 2017.

\bibitem[GLS16]{girouard2016steklov}
Alexandre Girouard, Richard~S Laugesen, and BA~Siudeja.
\newblock Steklov eigenvalues and quasiconformal maps of simply connected planar domains.
\newblock {\em Archive for Rational Mechanics and Analysis}, 219:903--936, 2016.

\bibitem[GP24]{girouard2024large}
Alexandre Girouard and Panagiotis Polymerakis.
\newblock Large steklov eigenvalues under volume constraints.
\newblock {\em The Journal of Geometric Analysis}, 34(10):316, 2024.

\bibitem[HH22]{he_upper_2022}
Zunwu He and Bobo Hua.
\newblock Upper bounds for the {Steklov} eigenvalues on trees.
\newblock {\em Calculus of Variations and Partial Differential Equations}, 61(3):101, April 2022.

\bibitem[HH23]{han_steklov_2023}
Wen Han and Bobo Hua.
\newblock Steklov eigenvalue problem on subgraphs of integer lattices.
\newblock {\em Communications in Analysis and Geometry}, 31(2):343--366, 2023.

\bibitem[HHW17]{hua_first_2017}
Bobo Hua, Yan Huang, and Zuoqin Wang.
\newblock First eigenvalue estimates of {Dirichlet}-to-{Neumann} operators on graphs.
\newblock {\em Calculus of Variations and Partial Differential Equations}, 56(6):178, December 2017.

\bibitem[HM20]{hassannezhad_higher_2020}
Asma Hassannezhad and Laurent Miclo.
\newblock Higher order {Cheeger} inequalities for {Steklov} eigenvalues.
\newblock {\em Annales Scientifiques de l'École Normale Supérieure. Quatrième Série}, 53(1):43--88, 2020.

\bibitem[Kar17]{karpukhin2017BoundsLaplaceSteklov}
Mikhail Karpukhin.
\newblock Bounds between {Laplace} and {Steklov} eigenvalues on nonnegatively curved manifolds.
\newblock {\em Electronic Research Announcements in Mathematical Sciences}, 24:100--109, 2017.
\newblock tex.fjournal: Electronic Research Announcements in Mathematical Sciences tex.zbl: 1404.35305 tex.zbmath: 6997326.

\bibitem[KKP14]{karpukhin2014multiplicity}
Mikhail Karpukhin, Gerasim Kokarev, and Iosif Polterovich.
\newblock Multiplicity bounds for steklov eigenvalues on riemannian surfaces.
\newblock {\em Annales de l'Institut Fourier}, 64(6):2481--2502, 2014.

\bibitem[KLMN04]{krauthgamer2004measured}
Robert Krauthgamer, James~R Lee, Manor Mendel, and Assaf Naor.
\newblock Measured descent: A new embedding method for finite metrics.
\newblock In {\em 45th Annual IEEE Symposium on Foundations of Computer Science}, pages 434--443. IEEE, 2004.

\bibitem[LS10]{lee2010genus}
James~R Lee and Anastasios Sidiropoulos.
\newblock Genus and the geometry of the cut graph.
\newblock In {\em Proceedings of the twenty-first annual ACM-SIAM symposium on Discrete Algorithms}, pages 193--201. SIAM, 2010.

\bibitem[LZ24]{lin_first_2024}
Huiqiu Lin and Da~Zhao.
\newblock The first {Steklov} eigenvalue of planar graphs and beyond, July 2024.

\bibitem[Mol06]{molitierno_algebraic_2006}
Jason~J. Molitierno.
\newblock On the algebraic connectivity of graphs as a function of genus.
\newblock {\em Linear Algebra and its Applications}, 419(2-3):519--531, 2006.

\bibitem[Per19]{perrin_lower_2019}
Hélène Perrin.
\newblock Lower bounds for the first eigenvalue of the {Steklov} problem on graphs.
\newblock {\em Calculus of Variations and Partial Differential Equations}, 58(2):12, 2019.

\bibitem[Per21]{perrin_isoperimetric_2021}
Hélène Perrin.
\newblock Isoperimetric upper bound for the first eigenvalue of discrete {Steklov} problems.
\newblock {\em The Journal of Geometric Analysis}, 31(8):8144--8155, 2021.

\bibitem[RS04]{ROBERTSON2004325}
Neil Robertson and P.D. Seymour.
\newblock Graph minors. xx. wagner's conjecture.
\newblock {\em Journal of Combinatorial Theory, Series B}, 92(2):325--357, 2004.
\newblock Special Issue Dedicated to Professor W.T. Tutte.

\bibitem[SSSV96]{shahrokhi1996crossing}
Farhad Shahrokhi, Ondrej S{\`y}kora, L{\'a}szl{\'o}~A Sz{\'e}kely, and Imrich Vrt'o.
\newblock The crossing number of a graph on a compact 2-manifold.
\newblock {\em Advances in Mathematics}, 123(2):105--119, 1996.

\bibitem[ST07]{spielman_spectral_2007}
Daniel~A. Spielman and Shang-Hua Teng.
\newblock Spectral partitioning works: {Planar} graphs and finite element meshes.
\newblock {\em Linear Algebra and its Applications}, 421(2-3):284--305, March 2007.

\bibitem[Tsc22]{tschanz2022upper}
L{\'e}onard Tschanz.
\newblock Upper bounds for steklov eigenvalues of subgraphs of polynomial growth cayley graphs.
\newblock {\em Annals of Global Analysis and Geometry}, 61(1):37--55, 2022.

\bibitem[Wei54]{weinstock_inequalities_1954}
Robert Weinstock.
\newblock Inequalities for a classical eigenvalue problem.
\newblock {\em Journal of Rational Mechanics and Analysis}, 3:745--753, 1954.

\bibitem[WX09]{wang2009sharp}
Qiaoling Wang and Changyu Xia.
\newblock Sharp bounds for the first non-zero stekloff eigenvalues.
\newblock {\em Journal of Functional Analysis}, 257(8):2635--2644, 2009.

\bibitem[YY17]{yang2017estimates}
Liangwei Yang and Chengjie Yu.
\newblock Estimates for higher steklov eigenvalues.
\newblock {\em Journal of Mathematical Physics}, 58(2), 2017.

\bibitem[ZL01]{zhang_kth_2001}
Xiaodong Zhang and Jiongsheng Li.
\newblock On the $k$-th largest eigenvalue of the {Laplacian} matrix of a graph.
\newblock {\em Acta Mathematicae Applicatae Sinica. English Series}, 17(2):183--190, 2001.

\end{thebibliography}
